\documentclass{lms}

\newtheorem{theorem}{Theorem}[section] 
\newtheorem{lemma}[theorem]{Lemma}     
\newtheorem{corollary}[theorem]{Corollary}

\newnumbered{assertion}{Assertion}    
\newnumbered{conjecture}{Conjecture}  
\newnumbered{definition}{Definition}
\newnumbered{hypothesis}{Hypothesis}
\newnumbered{remark}{Remark}
\newnumbered{note}{Note}
\newnumbered{observation}{Observation}
\newnumbered{problem}{Problem}
\newnumbered{question}{Question}
\newnumbered{algorithm}{Algorithm}
\newnumbered{example}{Example}
\newunnumbered{notation}{Notation} 

\title{Extending small arcs to large arcs} 

 \usepackage{amsmath}
\usepackage{amssymb}
\usepackage{graphicx}
\usepackage{amscd}

\author{Simeon Ball}

\classno{51E21, 15A03, 94B05, 05B35}

\extraline{The author acknowledges the support of the project MTM2014-54745-P of the Spanish {\em Ministerio de Econom\'ia y Competitividad.}}

\begin{document}
\maketitle

\begin{abstract}
An arc is a set of vectors of the $k$-dimensional vector space over the finite field with $q$ elements ${\mathbb F}_q$, in which every subset of size $k$ is a basis of the space, i.e. every $k$-subset is a set of linearly independent vectors. 
Given an arc $G$ in a space of odd characteristic, we prove that there is an upper bound on the largest arc containing $G$. The bound is not an explicit bound but is obtained by computing properties of a matrix constructed from $G$. In some cases we can also determine the largest arc containing $G$, or at least determine the hyperplanes which contain exactly $k-2$ vectors of the large arc.
The theorems contained in this article may provide new tools in the computational classification and construction of large arcs. The article also simplifies some of the proofs of the results found in \cite{Ball2012},  \cite[Chapter 7]{Ball2015}, \cite{BdB2012} and \cite{Chowdhury2015} and unifies the approach taken in those articles with that in \cite{BBT1990} and \cite{Segre1967}.
\end{abstract}


\section{Introduction} 

Let $\mathrm{V}_k({\mathbb F}_q)$ denote the $k$-dimensional vector space over ${\mathbb F}_q$, the finite field with $q$ elements.

An {\em arc} of $\mathrm{V}_k({\mathbb F}_q)$ is a set $S$ of vectors of $\mathrm{V}_k({\mathbb F}_q)$ in which every subset of size $k$ is a basis of $\mathrm{V}_k({\mathbb F}_q)$. Most authors define an arc, equivalently, as a set of points in the corresponding projective space. However, for the techniques that we wish to develop here it is more convenient to use the vector space. The set of columns of a generator matrix of a $k$-dimensional linear maximum distance separable (MDS) code over ${\mathbb F}_q$ is an arc of $\mathrm{V}_k({\mathbb F}_q)$ and vice-versa, so arcs and linear MDS codes are equivalent objects. As in coding theory, we define the {\em weight} of a vector to be the number of non-zero coordinates that it has.

We will assume throughout that $k\geqslant 3$.

Let $\det(v_1,\ldots,v_k)$ denote the determinant of the matrix whose $i$-th row is $v_i$, a vector of ${\mathrm V}_k({\mathbb F}_q)$. If $C=\{p_1,\ldots,p_{k-1} \}$ is an ordered set of $k-1$ vectors then we write
$$
\det(u,C)=\det(u,p_1,\ldots,p_{k-1}),
$$
where we evaluate the determinant with respect to a fixed canonical basis. 

Given an arc $G$ of $\mathrm{V}_k({\mathbb F}_q)$ and a non-negative integer $n \leqslant |G|-k$, we order the elements of $G$ arbitrarily and construct a matrix $\mathrm{M}_n$, in the following way. For each subset $E$ of $G$ of size $|G|-n$ and subset $A$ of $E$ of size $k-2$, we get a column of the matrix $\mathrm{M}_{ n}$, whose rows are indexed by subsets $C$ of $G$ of size $k-1$, where a $(C,(A,E))$ entry is
$$
\prod_{u \in G \setminus E} \det(u,C),
$$
if and only if $A \subset C$, and zero otherwise.

We shall prove the following theorem.

\begin{theorem} \label{wone}
If there is a vector of weight one in the column space of $\mathrm{M}_{ n}$ then $G$ cannot be extended to an arc of size $q+2k+n-1-|G|$.
\end{theorem}

If $q$ is odd then Theorem~\ref{wone} always provides at least some upper bound on the size of an arc $S$ containing $G$. To see this, consider $M_{|G|-k}$ and the columns for a fixed subset $E$ of $G$ of size $k$. Restricting to rows which are not the all-zero row vector, one obtains a copy of an inclusion matrix whose rows are indexed by the $(k-1)$-subsets of $E$ and whose columns are indexed by $(k-2)$-subsets of $E$. Since $E$ is fixed, the non-zero entries in any row are the same. Therefore, to calculate the rank we can divide out by this non-zero element and assume that each entry of the matrix is either $0$ or $1$. It's straightforward to verify that this matrix has rank $k$ if $q$ is odd (see \cite{Frankl1990} or \cite{Wilson1990} for a general formula for $p$-ranks of inclusion matrices) and hence full row rank, which implies there is a vector of weight one in the column space of $\mathrm{M}_{|G|-k}$. 

A weaker version of Theorem~\ref{wone} appears in \cite{Chowdhury2015}, where the condition on $\mathrm{M}_n$ is that it should have full row rank, although the stronger Theorem~\ref{wone} is also observed as a remark. There are many examples where $M_n$ does not have full row rank but where it does have a vector of weight one in its column space. For example, the arc of size $7$
$$
G=\{ [ \tau^0, 0, 0 ], [ 0, \tau^0, 0 ], 
  [ 0, 0, \tau^0 ], [ \tau^0, \tau^5, \tau^0 ], 
  [ \tau^0, \tau^8, \tau^9 ], [ \tau^0, \tau, \tau^5 ], 
  [ \tau^0, \tau^3, \tau ]  \},
$$
where $\tau$ is a primitive element of ${\mathbb F}_{11}$. The matrix $\mathrm{M}_{ 2}$ does not have full row rank (it has rank 20, whereas full row-rank would be 21), but it does have a vector of weight one in its column space. Therefore, Theorem~\ref{wone} implies that it cannot be extended to an arc of size $11$. It can be extended to an arc of size $10$.

We say that $(G,n)$ has {\em Property W} if for each subset $A$ of $G$ of size $k-2$, there exist subsets $C,C_1,\ldots,C_{|G|-n-k+1}$ of $G$ of size $k-1$ containing $A$, such that the column space of $\mathrm{M}_{ n}$ contains a vector (of weight two) with non-zero coordinates at only $C$ and $C_i$ for each $i=1,\ldots,|G|-n-k+1$.

We define a {\em co-secant} to an arc $S$ of $\mathrm{V}_k({\mathbb F}_q)$ to be a hyperplane containing precisely $k-2$ vectors of $S$. An arc is {\em complete} if it is not contained in a larger arc, and an arc is {\em uniquely completable} if there is only one complete arc containing it.

We shall also prove the following theorem.

\begin{theorem} \label{wtwo}
If $(G,n)$ has Property $W$ and $G$ can be extended to an arc $S$ of size $q+2k+n-1-|G|$ then the co-secants to $S$ containing only points of $G$ are determined by $G$.
\end{theorem}

Theorem~\ref{wtwo} has the following corollary.

\begin{corollary} \label{ctwo}
If $\mathrm{M}_{ n}$ has rank one less than full row rank and $G$ can be extended to an arc $S$ of size $q+2k+n-1-|G|$ then the co-secants to $S$ containing only points of $G$ are determined by $G$.
\end{corollary} 

\begin{proof}
If $G$ can be extended to an arc $S$ of size $q+2k+n-1-|G|$ then by Theorem~\ref{wone} there are no vectors of weight one in the column space of $\mathrm{M}_{ n}$. Since the rank of $\mathrm{M}_{ n}$ is one less than full row rank, using column operations one can transform the matrix $\mathrm{M}_{ n}$, into a matrix that has the identity matrix (of size one less than the number of rows) in the top left hand corner. By taking one of these columns, or a linear combination of two of them, we have that every vector of weight two is in the column space of $\mathrm{M}_{ n}$, so Property $W$ holds.
\end{proof}

There is an inclusion-reversing duality between the $r$-dimensional subspaces of ${\mathrm V}_k({\mathbb F}_q)$ and the $(k-r)$-dimensional subspaces. Under this duality the set of co-secants to an arc $S$ of size $q+k-1-t$ is a set of ${|S| \choose k-2}t$ vectors. In \cite{Segre1967} and \cite{BBT1990} respectively, Segre (the case $k$ is three) and Blokhuis, Bruen and Thas prove that this set of vectors is contained in an algebraic hypersurface $\phi_S$ of degree $t$ if $q$ is even and of degree $2t$ if $q$ is odd. Moreover, $\phi_S$ can be constructed from a sub-arc $G$ of $S$ of size $k+t-1$ if $q$ is even and $k+2t-1$ if $q$ is odd, if one knows all the co-secants to $S$ containing only points of $G$, see Section~\ref{appendix}. Therefore, Theorem~\ref{wtwo} can be strengthened to the following theorem.

\begin{theorem} \label{wtwo2}
If $(G,n)$ has Property $W$ and $G$ can be extended to an arc $S$ of size $q+2k+n-1-|G|$ then $\phi_S$ is determined by $G$, provided that $2n \geqslant |G|-k-1$ in the case that $q$ is odd.
\end{theorem}

Theorem~\ref{wtwo2} has the following corollary, the proof of which is identical to the proof of Corollary~\ref{ctwo}.

\begin{corollary} \label{ctwo2}
If $\mathrm{M}_{ n}$ has rank one less than full row rank and $G$ can be extended to an arc $S$ of size $q+2k+n-1-|G|$ then $\phi_S$ is determined by $G$, provided that $2n \geqslant |G|-k-1$ in the case that $q$ is odd.
\end{corollary} 

We will prove in Section~\ref{appendix} that $G$ never satisfies the hypotheses of Theorem~\ref{wone}, Theorem~\ref{wtwo} and Theorem~\ref{wtwo2} when $q$ is even, see Theorem~\ref{forgetqeven}, apart from in the trivial case that $|G|=k+n$ in which case the hypotheses of Theorem~\ref{wtwo} and Theorem~\ref{wtwo2} are satisfied.
In \cite{Segre1967} and \cite{BBT1990} respectively, Segre (the case $k$ is three) and Blokhuis, Bruen and Thas prove that if $q$ is odd and $G$ has size at least $k-1+(2q/3)$  or if $q$ is even and $G$ has size at least $k+(q/2)$ then $G$ is uniquely completable. Thus, Theorem~\ref{wtwo} and Theorem~\ref{wtwo2} are of interest only when $q$ is odd and $G$ is smaller than $k-1+(2q/3)$.

To illustrate the applicability of Theorem~\ref{wtwo} and Theorem~\ref{wtwo2} we first consider some examples. It is perhaps surprising that $G$ can be relatively small and still have Property $W$.
Let
$$
G=\{ [ \epsilon^0, 0 ,0 ], [ 0, \epsilon^0, 0 ], 
[ 0, 0, \epsilon^0 ], [ \epsilon^0, \epsilon^2, \epsilon^3 ], 
[ \epsilon^0, \epsilon^6, \epsilon^4 ], 
[ \epsilon^0, \epsilon^{9}, \epsilon^{9} ], 
[ \epsilon^0, \epsilon^4, \epsilon^6 ], [ \epsilon^0, \epsilon^{11}, \epsilon^5 ], 
[ \epsilon^0, \epsilon^0, \epsilon^8 ]\},
$$
where $\epsilon$ is a primitive element of ${\mathbb F}_{13}$. Then $(G,3)$ has Property $W$, so Theorem~\ref{wtwo} implies that if it can be extended to an arc $S$ of $\mathrm{V}_3({\mathbb F}_{13})$ of size $12$ (and it can) then $\phi_S$ is determined by $G$. 

Let
$$
G=\{
 [ \epsilon^0, 0, 0], [ 0, \epsilon^0, 0 ], 
 [ 0, 0, \epsilon^0 ], [ \epsilon^0, \epsilon^{10}, \epsilon^2 ], 
 [ \epsilon^0, \epsilon^2, \epsilon^{11} ], [ \epsilon^0, \epsilon^9, \epsilon^4 ] \},
$$
where $\epsilon$ is a primitive element of ${\mathbb F}_{13}$. Then $(G,2)$ has Property $W$, so Theorem~\ref{wtwo} implies that if it can be extended to an arc $S$ of $\mathrm{V}_3({\mathbb F}_{13})$ of size $14$ (and it can) then $\phi_S$ is determined by $G$. In this case $\phi_S$ also determines $S$. Note that by Segre's theorem \cite{Segre1955a}, we know that $S$ is a conic. The small arc $G$ also completes to arcs of size $9$, $10$ and $12$, so it is not uniquely completable. 

In \cite{Ball2012}, it is shown that if $k\leqslant p$, where $p$ is the prime such that $q$ is a power of $p$, then there are no arcs of size $q+2$. This follows from Theorem~\ref{wone} by considering the $p$-rank formula for inclusion matrices from \cite{Frankl1990} or \cite{Wilson1990}. If $k\leqslant p$ then for any arc $G$ of size $2k-3$ the matrix $\mathrm{M}_{ 0}$ has full row-rank and therefore a vector of weight one in its column space. In \cite{BdB2012}, it is shown that if $k\leqslant 2p-2$, where $q$ is a non-prime prime power, then there are no arcs of size $q+2$. Indeed it can be shown that, see \cite{Chowdhury2015}, for an arc $G$ of size $2k-2$ the matrix $\mathrm{M}_{ 1}$ has a vector of weight one in its column space, so Theorem~\ref{wone} implies that if $k\leqslant 2p-2$ then there are no arcs of size $q+2$. Computational evidence for small $k$ and small $p$ suggests that the following is true.

\begin{conjecture} \label{weakmds}
If $k \leqslant p+n(p-2)$ and $G$ has size $2k-3+n$ then the matrix $\mathrm{M}_{ n}$ has a vector of weight one in its column space. 
\end{conjecture}

If Conjecture~\ref{weakmds} is true then Theorem~\ref{wone} would imply that there are no arcs of size $q+2$ for $k \leqslant (pq-2q+6p-10)/(2p-3)$. This would verify the MDS conjecture  for these $k$. For more on the MDS conjecture, see for example \cite{Ball2012}, \cite{HS2001}, \cite{MS1977} or \cite{Vardy2006}. A version of Conjecture~\ref{weakmds} in which $\mathrm{M}_{ n}$ is conjectured to have full row-rank appears in \cite{Chowdhury2015}.
Since there are examples of arcs $G$ of size $2k-3+n$ with $k>p+n(p-2)$ for which $\mathrm{M}_{ n}$ does not have vectors of weight one in its column space, this would appear to put a limit on these methods for verifying the MDS conjecture in its entirety. To verify the MDS conjecture one must show that there are no arcs of size $q+2$ for $4 \leqslant k \leqslant (q+2)/2$, so one would fall short.  

Theorem~\ref{wone}, Theorem~\ref{wtwo} and Theorem~\ref{wtwo2} may be of some use in classifying or at least constructing large arcs computationally. See \cite{Coolsaet2015}, \cite{CS2011} and \cite{Keri2006} for recent computational results regarding arcs. To classify arcs of size $q+k-r$ one would need to classify arcs of size $k+r$. If one could classify arcs of size $k+r$ then one could quickly check for each arc to see if $\mathrm{M}_1$ has a vector of weight one in the column space, for each projectively distinct arc $G$. In the positive case, this would then rule out the possibility that $G$ can be extended to an arc of size $q+k-r$. In the negative case, one can then extend the arc to an arc $H$ of size $k+r+1$ and check to see if $\mathrm{M}_2$ (calculated using $H$) has a vector of weight one in the column space. In the positive case, this would then rule out the possibility that $H$ can be extended to a arc of size $q+k-r$. This should dramatically reduce the set of possible sub-arcs of arcs of size $q+k-r$. Those small arcs which cannot be ruled out as possibly extending to a large arc of course may well extend. In this case Theorem~\ref{wtwo} and Theorem~\ref{wtwo2} come into play. If $(G,1)$ or $(H,2)$ etc., has property $W$ then Theorem~\ref{wtwo} and possibly Theorem~\ref{wtwo2} applies. Knowing $G$ and the co-secants to $S$ containing the points of $G$ may well be enough to determine $S$ and even if it doesn't, it will certainly drastically reduce the possible vectors which might extend $G$.

For example, consider the following arc of size $11$ of $\mathrm{V}_6({\mathbb F}_{81})$
$$
G=\{ [ \rho^0, 0, 0, 0, 0, 0 ], 
  [ 0, \rho^0, 0, 0, 0, 0 ], 
  [ 0, 0, \rho^0, 0, 0, 0 ], 
  [ 0, 0, 0, \rho^0, 0, 0 ], 
  [ 0, 0, 0, 0, \rho^0, 0 ], 
  $$
  $$
  [ 0, 0, 0, 0, 0, \rho^0 ], 
  [ \rho^0, \rho^0, \rho^0, \rho^0, \rho^0, \rho^0 ], 
  [ \rho^0, \rho^{58}, \rho^{41}, \rho^{14}, \rho^{54}, \rho^{48} ],
  $$
  $$ 
  [ \rho^0, \rho^{25}, \rho^{55}, \rho^{43}, \rho^{74}, \rho^{58} ], 
  [ \rho^0, \rho, \rho^{66}, \rho^{22}, \rho^{42}, \rho^{65} ], 
  [ \rho^0, \rho^{76}, \rho^{44}, \rho^{21}, \rho^{43}, \rho^5 ] \},
$$
where $\rho$ is a primitive element of ${\mathbb F}_{81}$. The matrix $\mathrm{M}_1$ has rank $461$ and full row rank would be $462$. It does not have a vector of weight one in its column space, so Theorem~\ref{wone} does not apply. However, Corollary~\ref{ctwo} does apply. Suppose that $G$ extends to an arc $S$ of size $82$. From a vector which is a basis of the null space of the row space of $\mathrm{M}_1$ we can calculate $f_A$ (see the next section) for every subset $A$ of $G$ of size $4$. These polynomials of degree $4$ must all be fully reducible into linear factors. Moreover, for each linear form $\alpha$, which is a factor of $f_A$ for some $A$, the hyperplane $\ker \alpha$ contains only the points $A$ of $S$ and no other points of $S$. This severely restricts how one can extend $G$.

Even if the classification of arcs of size $k+r$ for a certain $q$ is infeasible computationally, one may be able to construct new examples of size $q+k-r$. By identifying the small arcs which can appear as sub-arcs of an arc of size $q+k-r$ one can apply Theorem~\ref{wtwo} and Theorem~\ref{wtwo2}, as explained in the previous paragraph.

Theorem~\ref{wone}, Conjecture~\ref{weakmds}, and their applications to computationally classifying and constructing large arcs are joint work with Ameera Chowdhury.  The results in Section 2 and Section 3, which are simplifications of previous results from \cite{Ball2012}, \cite{Ball2015}, and \cite{BdB2012}, is also joint work with Ameera Chowdhury. She has written a separate exposition of these results in \cite{Chowdhury2015}.

\section{The functions $f_A$}

If $C=\{p_1,\ldots,p_{k-1} \}$ is an ordered set of $k-1$ vectors then we write
$$
\det(u,C)=\det(u,p_1,\ldots,p_{k-1}),
$$
and if $A=\{p_1,\ldots,p_{k-2} \}$ is an ordered set of $k-2$ vectors then we write
$$
\det(u,v,A)=\det(u,v,p_1,\ldots,p_{k-2}),
$$
where we evaluate the determinant with respect to a fixed canonical basis. This defines a bilinear form on $\mathrm{V}_k({\mathbb F}_q) \times \mathrm{V}_k({\mathbb F}_q)$ by
$$
d_A(u,v)=\det (u,v,A).
$$

Let $S$ be an arc of $\mathrm{V}_k({\mathbb F}_q)$. We order the elements of $S$ arbitrarily and maintain this order throughout, unless otherwise stated. Let $A$ and $B$ be subsets of the ordered set $S$. If we write $A,B$ in place of $A \cup B$ then this means order the elements of $A$ first and then order the elements of $B$.


Let $A$ be a subset of $S$ of size $k-2$.

\begin{lemma} \label{minusone}
The bilinear form $d_A$ is alternating and in particular $d_{A}(u,v)=-d_{A}(v,u)$.
\end{lemma}

\begin{proof}
By the definition of $d_A$ we have $d_A(u,u)=\det(u,u,A)=0$ and
$$
d_{A}(u,v)=\det(u,v,A)=-\det(v,u,A)=-d_{A}(v,u).
$$
\end{proof}

Let $t=q+k-1-|S|$.

\begin{lemma}
There are $t$ hyperplanes which contain the vectors of $A$ and no other vectors of $S$.
\end{lemma}

\begin{proof}
There are $q+1$ hyperplanes containing the subspace spanned by the $k-2$ vectors of $A$. Since $S$ is an arc, a hyperplane can contain at most $k-1$ vectors of $S$. Therefore, there are $|S|-|A|$ of them which contain one more vector of $S$ and so $q+1-|S|+k-2$ of them which contain no more vectors of $S$.
\end{proof}

Let $\alpha_1,\ldots,\alpha_t$ be pairwise linearly independent forms with the property that $\ker \alpha_i \cap S=A$. Define
$$
f_A(x)=\prod_{i=1}^t \alpha_i(x),
$$
a function from $\mathrm{V}_k({\mathbb F}_q)$ to ${\mathbb F}_q$.

\begin{lemma} \label{interp}
If $ E$ is a subset of $S$ of size $t+k-1$ containing $A$ then
$$
f_A(x)=\sum_{e \in E \setminus A} f_A(e) \prod_{u \in E \setminus (A\cup \{e\})} \frac{d_{A}(u,x)}{d_{A}(u,e)}.
$$
\end{lemma}

\begin{proof}
With respect to a basis $B$ whose last $k-2$ elements are $A$, $f_A$ is a homogeneous polynomial in two variables, and so is
$$
\sum_{e \in E \setminus A} f_A(e) \prod_{u \in E \setminus (A\cup \{e\})} \frac{d_{A}(u,x)}{d_{A}(u,e)}.
$$

These two polynomials are equal when evaluated at an element of $E \setminus A$. Two homogeneous polynomials of degree $t$ in two variables which are equal at $t+1$ linearly independent points are the same, since their difference is a homogeneous polynomials in two variables of degree at most $t$ and can be zero at at most $t$ linearly independent points. Note that it follows from the arc property that any two points of $E \setminus A$ are linearly independent, even after deleting the $k-2$ coordinates corresponding to the elements of $A$ in the basis.
\end{proof}

\begin{lemma} \label{thesumiszero}
If $ E$ is a subset of $S$ of size $t+k$ containing $A$ then
$$
\sum_{e \in E \setminus A} f_A(e) \prod_{u \in E \setminus (A\cup \{e\})} d_{A}(u,e)^{-1}=0.
$$
\end{lemma}

\begin{proof}
Suppose that $v \in E \setminus A$ and apply Lemma~\ref{interp} with $E$ replaced by $E\setminus \{ v \}$,
$$
f_A(v)=\sum_{e \in E \setminus (A \cup \{v\})} f_A(e) \prod_{u \in E \setminus (A\cup \{v, e\})} \frac{d_{A}(u,v)}{d_{A}(u,e)}.
$$
Dividing by 
$$
\prod_{u \in E \setminus (A\cup \{v\})} d_{A}(u,v)
$$ 
gives
$$
f_A(v)\prod_{u \in E \setminus (A\cup \{v\})} d_{A}(u,v)^{-1}=\sum_{e \in E \setminus (A \cup \{v\})} f_A(e) \frac{d_{A}(v,e)}{d_{A}(e,v)}\prod_{u \in E \setminus (A\cup \{e\})} d_{A}(u,e)^{-1},
$$
and so the lemma follows from Lemma~\ref{minusone}.
\end{proof}

The aim of the following section is to show that we can multiply the equation in Lemma~\ref{thesumiszero} by an element of ${\mathbb F}_q$, dependent on $A$, so that the terms depend only on $C=A \cup \{ e \}$ and not on $A$. This implies that we will get an equation for each $(k-2)$-subset of $E$ whose ``variables'' depend only on the $(k-1)$-subsets of $E$.

\section{A set of equations associated with an arc}

The following lemma is called the co-ordinate free version of Segre's lemma of tangents, proved in \cite{Ball2012} and also \cite[Lemma 7.15]{Ball2015}. In order that this article be self-contained we include a proof.


\begin{lemma} \label{segre1}
Let $D$ be a subset of $S$ of size $k-3$ and let $\{x,y,z \}$ be a subset of $S \setminus D$. Interchanging $x$ and $y$ in
$$
\frac{f_{D \cup \{x \}}(y)f_{D \cup \{z \}}(x)}{f_{D \cup \{x \}}(z)}
$$
changes the sign by $(-1)^{t+1}$.
\end{lemma}

\begin{proof}
Let $B= \{x,y,z \} \cup  D$. Since $B$ is a subset of $S$ of size $k$, it is a basis of $\mathrm{V}_k({\mathbb F}_q)$.

There are $q+1$ hyperplanes containing $\langle z,D\rangle$, since it is a $(k-2)$-dimensional subspace of $\mathrm{V}_k({\mathbb F}_q)$. We start off by identifying these $q+1$ hyperplanes.

Suppose that $u \in S \setminus B$ and that $(u_1,\ldots,u_k)$ are the coordinates of $u$ with respect to the basis $B$. The hyperplane $\langle u,z,D \rangle$ is
$$
\ker (u_2X_1-u_1X_2),
$$
since $\{z \} \cup D$ is the set of the last $k-2$ vectors of the basis $B$. For each $ u \in S \setminus B$, since $S$ is an arc, we have a distinct hyperplane $\langle u,z,D \rangle$, and so $|S \setminus B|=q-1-t$ of them in all.

Suppose that the function $f_{D \cup \{z\}}$ is
$$
f_{D \cup \{z\}}(u)=\prod_{i=1}^t \alpha_i(u),
$$
where $\ker \alpha_i \cap S=D\cup \{z \}$ and $\alpha_1,\ldots,\alpha_t$ are pairwise linearly independent linear forms.

With respect to the basis $B$, the linear form $\alpha_i(X)$ is
$$
\alpha_i(X)=\alpha_{i1}X_1+\alpha_{i2}X_2,
$$
for some $\alpha_{i1}, \alpha_{i2} \in {\mathbb F}_q$. Since $\ker \alpha_i \supset D \cup \{z \}$, this gives us a further $t$ hyperplanes containing $\langle z,D \rangle$.

The other two hyperplanes are $\ker X_1 = \langle y,z,D \rangle$ and $\ker X_2=\langle x,z,D \rangle$.

The $q-1$ hyperplanes containing $\langle z,D \rangle$, and not containing $x$ or $y$, are
$$
\ker (aX_1+X_2),
$$
where $a \in {\mathbb F}_q \setminus \{0 \}$. Therefore,
$$
\prod_{i=1}^t \frac{\alpha_{i1}}{\alpha_{i2}} \prod_{u \in S \setminus B} \frac{(-u_2)}{u_1}=-1,
$$
since it is the product of all non-zero elements of ${\mathbb F}_q$, which is $-1$.

With respect to the basis $B$, $x$ has coordinates $(1,0,\ldots,0)$, and so
$$
f_{D \cup \{z \} }(x)=f_{D\cup \{z \} } ((1,0,\ldots,0))=\prod_{i=1}^t \alpha_{i1}.
$$
Similarly $$f_{D \cup \{z\} }(y)=\prod_{i=1}^t \alpha_{i2},$$ so the equation above implies
$$
f_{D \cup \{z\} } (y)\prod_{u \in S \setminus B} u_1=(-1)^{t+1}  f_{D \cup \{z\}}(x) \prod_{u \in S \setminus B} u_2.
$$
Repeating the above, switching $y$ and $z$ gives,
$$
f_{D \cup \{y\} } (z)\prod_{u \in S \setminus B} u_1=(-1)^{t+1}  f_{D \cup \{y\}}(x) \prod_{u \in S \setminus B} u_3.
$$
And switching $x$ and $y$ gives,
$$
f_{D \cup \{x\} } (z)\prod_{u \in S \setminus B} u_2=(-1)^{t+1}  f_{D \cup \{x\}}(y) \prod_{u \in S \setminus B} u_3.
$$
Combining these three equations gives,
$$
f_{D \cup \{x \}}(y)f_{D \cup \{y \}}(z)f_{D \cup \{z \}}(x) =(-1)^{t+1}f_{D \cup \{x \}}(z)f_{D \cup \{y \}}(x)f_{D \cup \{z \}}(y),
$$
since 
$$
\prod_{u \in S \setminus B} u_1u_2 u_3 \neq 0.
$$
Thus, we have
$$
\frac{f_{D \cup \{x \}}(y)f_{D \cup \{z \}}(x)}{f_{D \cup \{x \}}(z)}=(-1)^{t+1}\frac{f_{D \cup \{y \}}(x)f_{D \cup \{z \}}(y)}{f_{D \cup \{y \}}(z)}
$$

\end{proof}

Let $F$ be the subset of the first $k-2$ elements of $S$ with respect to the ordering of $S$.

For a subset $A$ of $S$ of size $k-2$, let
$$
\alpha_A=(-1)^{(r+s)(t+1)} \prod_{i=1}^r \frac{f_{D \cup \{z_r,\ldots,z_{i},x_{i-1},\ldots,x_1 \} } (x_{i})}{f_{D \cup \{z_r,\ldots,z_{i+1},x_{i},\ldots,x_1 \} } (z_i)},
$$
where $D=A \cap F$, $A \setminus F=\{x_1,\ldots,x_{r} \}$, $F \setminus A=\{z_1,\ldots,z_r \}$ and $s$ is the number of transpositions required to order $(F \cap A,F\setminus A)$ as $F$.

For a subset $C$ of $S$ of size $k-1$, let
$$
\alpha_C=(-1)^{(r+s)(t+1)} f_{D\cup \{ x_r \ldots,x_1\}}(x_{r+1})\prod_{i=1}^r \frac{f_{D \cup \{z_r,\ldots,z_{i},x_{i-1},\ldots,x_1 \} } (x_{i})}{f_{D \cup \{z_r,\ldots,z_{i+1},x_{i},\ldots,x_1 \} } (z_i)},
$$
where $D=C \cap F$, $C \setminus F=\{x_1,\ldots,x_{r+1} \}$, $F \setminus C=\{z_1,\ldots,z_r \}$ and $s$ is the number of transpositions required to order $(F \cap C,F\setminus C)$ as $F$.

The following is from \cite[Lemma 7.19]{Ball2015}. Again, in order that this article be self-contained, we include a proof.

\begin{lemma} \label{atoc}
For a subset $A$ of $S$ of size $k-2$, and $e \in S \setminus A$,
$$
\alpha_{A \cup \{ e \}}=(-1)^{d(t+1)} \alpha_A f_A(e),
$$
where $d$ is the number of elements of $A$ that come after $e$ in the ordering.
\end{lemma}

\begin{proof}
If $e \not\in F$ then $F \setminus (A \cup \{ e \})=F \setminus A$ and $A \cap F=(A\cup \{ e \}) \cap F$ is immediate. We have to reorder the numerator of $\alpha_Af_A(e)$ so that it coincides with $\alpha_{A \cup \{ e \}}$. Then we can write $\alpha_{A \cup \{ e \}}$ in place of $\alpha_Af_A(e)$. By Lemma~\ref{segre1}, this changes the sign by 
$$
(-1)^{m(t+1)},
$$
where $m$ is the number of elements of $A \setminus F$ that come after $e$ in the ordering. 
Since $e \not\in F$, and the elements of $F$ come first in the ordering, $m$ is the number of elements of $A$ that come after $e$ in the ordering. Therefore, $m=d$ and this case is done.

If $e \in F$ then we have to reorder the denominator of $\alpha_A$ to move the $e \in F \setminus A$ so that it is $z_r$.  Then, up to getting the sign right, we are able to write $\alpha_{A \cup \{e \}}$ in place of $\alpha_Af_A(e)$, since the $f_A(e)$ cancels with one in the denominator. Note that $e \in F \cap (A \cup \{ e \})$. 

This reordering, according to Lemma~\ref{segre1}, changes the sign by
$$
(-1)^{m_1(t+1)},
$$
where $m_1$ is the number of elements of $F \setminus (A \cup \{ e\})$ which come after $e$ in the ordering.

Note that $|F \setminus (A \cup \{ e \})|=|F \setminus A|-1$, so we have to decrease $r$ by $1$ when we replace $\alpha_Af_A(e)$ by $\alpha_{A \cup \{e \}}$, while $s$ increases by the number of elements of $F\setminus (A \cup \{ e\})$ which come before $e$ in the ordering plus the number of elements of $F \cap A$ that come after $e$ in the ordering. So, all in all, the sign changes by 
$$
(-1)^{(|F\setminus (A \cup \{ e\})|+m_2-1)(t+1)}
$$
where $m_2$ is the number of elements of $F \cap A=(F\setminus \{ e\}) \cap A$ that come after $e$ in the ordering. Since the elements of $F$ come first in the ordering and $e \in F$, we have that $m_3=|(F\setminus \{ e\}) \cap A|-m_2$ is the number of elements of $A$ that come before $e$ in the ordering. Therefore, the sign changes by
$$
(-1)^{(|F \setminus \{ e \}|+m_3-1)(t+1)}.
$$
The lemma follows since $d=k-2-m_3$ and $|F\setminus \{ e \}|=k-3$.
\end{proof}

\begin{lemma} \label{theeqn}
Let $S$ be an arbitrarily ordered arc of size $q+k-1-t$ and let $E$ be a subset of $S$ of size $k+t$. For any subset $A$ of $E$ of size $k-2$,
$$
\sum \alpha_C \prod_{u \in E \setminus C} \det(u,C)^{-1}=0,
$$
where the sum runs over the subsets $C$ of $E$ of size $k-1$ containing $A$.
\end{lemma}

\begin{proof}
By Lemma~\ref{thesumiszero}, since $E$ is a subset of $S$ of size $t+k$ containing $A$, 
$$
\sum_{e \in E \setminus A} f_A(e) \prod_{u \in E \setminus (A\cup \{e\})} \det(u,e,A)^{-1}=0.
$$
Observe that 
$$
\det(u,e,A)=(-1)^{k-2}\det(u,A,e)=(-1)^{k-2+d} \det(u,A \cup \{ e \}),
$$
where $d$ is the number of elements of $A$ which come after $e$ in the ordering. Since there are $t+1$ terms in the product, when we multiply by $\alpha_A$ and apply Lemma~\ref{atoc}, the lemma follows.
\end{proof}

\section{Proofs of Theorem~\ref{wone}, Theorem~\ref{wtwo} and Theorem~\ref{wtwo2}}

Let $n$ be a non-negative integer and let $G$ be an arc of $\mathrm{V}_k({\mathbb F}_q)$ of size at least $k+n$. Order the elements of $G$ arbitrarily and let $F$ be the set of the first $k-2$ vectors of $G$. 

Recall that we defined the matrix $\mathrm{M}_{ n}$ as follows. For each subset $E$ of $G$ of size $|G|-n$ and subset $A$ of $E$ of size $k-2$, we get a column of the matrix $\mathrm{M}_{ n}$, whose rows are indexed by subsets $C$ of $G$ of size $k-1$, where a $(C,(A,E))$ entry is
$$
\prod_{u \in G \setminus E} \det(u,C),
$$
if and only if $A \subset C$ and zero otherwise.

Let $v_G$ be a vector whose coordinates are indexed by the subsets $C$ of $G$ of size $k-1$ and whose $C$ coordinate is 
$$
\alpha_C \prod_{z \in G \setminus C} \det(z,C)^{-1}.
$$ 

\begin{lemma} \label{upsolution}
If $G$ can be extended to an arc of size $q+2k+n-1-|G|$ then
$v_G\mathrm{M}_{ n}=0$.
\end{lemma}

\begin{proof}
Suppose that $G$ can be extended to $S$, an arc of size $q+2k+n-1-|G|$. Let $t=|G|-k-n$. For each subset $E$ of $G$ of size $k+t$, the equation in Lemma~\ref{theeqn} for a subset $A$ of $E$ of size $k-2$, is the equation obtained by multiplying $v_G$ with the $(A,E)$ column of $\mathrm{M}_{ n}$.
\end{proof}

\begin{proof} (of Theorem~\ref{wone}).
If there is a vector of weight one in the column space of $\mathrm{M}_{ n}$ then the scalar product of this vector with $v_G$, according to Lemma~\ref{upsolution}, gives us the equation 
$$
\alpha_C \prod_{z \in G \setminus C} \det(z,C)^{-1}=0,
$$  
for some subset $C$ of $G$ of size $k-1$. This is a contradiction, since all terms in this product are non-zero.
\end{proof}

\begin{proof} (of Theorem~\ref{wtwo})
Let $A$ be a subset of $G$ of size $k-2$ and let $S$ be an arc of size $q+2k+n-1-|G|$ containing $G$. As before, let $t=q+k-1-|S|$, so $|G|=k+t+n$. 

Since $(G,n)$ has property $W$ there are elements $x,y_1,\ldots,y_{t+1}$ of $G$ and vectors $u_1,\ldots,u_{t+1}$ in the column space of $\mathrm{M}_{ n}$, where $u_i$ has non-zero $C$ coordinates if and only if $C=A\cup \{x\}$ or $C=A \cup \{y_i \}$. By Lemma~\ref{upsolution}, the scalar product of $u_i$ with $v_G$, gives the equation 
$$
\alpha_{A\cup \{x\}} \prod_{z \in G \setminus A \cup \{x\}} \det(z,A \cup \{x \})^{-1}= a_i
\alpha_{A\cup \{y_i\}} \prod_{z \in G \setminus A \cup \{y_i\}} \det(z,A \cup \{y_i \})^{-1},
$$
for some $a_i \in {\mathbb F}_q$.

By Lemma~\ref{atoc}, this determines 
$$
\frac{f_A(y_i)}{f_A(x)},
$$
so this quantity is determined by $G$. Hence, $G$ determines the value of 
$$
\frac{f_A(X)}{f_A(x)}
$$
at $t+1$ linearly independent points. Since $f_A(X)$ is a homogeneous polynomial of degree $t$, this determines $f_A(X)$, so $G$ determines $f_A(X)$. Each factor of $f_A(X)$ must be a linear form $\alpha$, where $\ker \alpha$ is a hyperplane intersecting $S$ in $A$. Therefore, $G$ determines all these hyperplanes, which is what we wanted to prove.
\end{proof}

\begin{proof} (of Theorem~\ref{wtwo2})
The algebraic hypersurface $\phi_S$ can be constructed from a subset $E$ of $S$ of size $k+t-1$ if $q$ is even and size $k+2t-1$ if $q$ is odd, if one knows all the co-secants to $S$ containing $A$ for every subset $A$ of $E$, see \cite{BBT1990} or Section~\ref{appendix}. In Section~\ref{appendix} we give an explicit description of $\phi_S$. The condition $2n \geqslant |G|-k-1$ implies $|G| \geqslant |E|$ if $q$ is odd, so by Theorem~\ref{wtwo} we can construct $\phi_S$.
\end{proof}

\section{The algebraic hypersurface associated with an arc} \label{appendix}

In this section we explicitly construct the algebraic hypersurface $\phi_S$ associated with an arc $S$ of $\mathrm{V}_k({\mathbb F}_q)$, introduced in \cite{BBT1990}.

As before, let $t=q+k-1-|S|$. Let $E$ be a subset of $S$ of size $k+t-1$ is $q$ is even and $k+2t-1$ if $q$ is odd. To able to find such an $E$, and therefore construct $\phi_S$, this imposes a lower bound on the size of $S$.

For $q$ even, define a polynomial in $k-1$ vector variables, so $k(k-1)$ indeterminates,
$$
\phi_S(Y_1,\ldots,Y_{k-1})=\sum_{C} \alpha_C \prod_{z \in E \setminus C} \frac{\det(z,Y_1,\ldots,Y_{k-1})}{\det(z,C)},
$$
where the sum runs over all subsets $C$ of size $k-1$ of $E$. 

For $q$ odd, define a polynomial in $k-1$ vector variables,
$$
\phi_S(Y_1,\ldots,Y_{k-1})=\sum_{C}  \alpha_C^2 \prod_{z \in E \setminus C} \frac{\det(z,Y_1,\ldots,Y_{k-1})}{\det(z,C)} ,
$$
where the sum runs over all subsets $C$ of size $k-1$ of $E$. 

Although $\phi_S$ is defined as a polynomial in $k-1$ vector variables, a simple change of variables shows that in fact it can be written as a polynomial in $k$ indeterminates. Let 
$$
Z_i=(-1)^{i-1}\det(Y_1,\ldots,Y_{k-1}),
$$
where the $i$-th coordinate of $Y_j$ has been deleted, so the determinant is of a $(k-1) \times (k-1)$ matrix. Then 
$$
\phi_S=\phi_S(Z_1,\ldots,Z_k).
$$
Let $\{ c_1,\ldots,c_{k-1} \}$ be a set of $k-1$ linearly independent vectors of $\mathrm{V}_k({\mathbb F}_q)$. With $Y_j=c_j$ for $j=1,\ldots,k-1$, this defines $z_i=Z_i$, for $i=1,\ldots,k$. The vector $(z_1,\ldots,z_k)$ is a vector in the dual space, dual to the hyperplane spanned by $\{ c_1,\ldots,c_{k-1} \}$. Suppose that $\{ c_1,\ldots,c_{k-1} \}$ spans a co-secant hyperplane to $S$. Then $(c_1,\ldots,c_{k-1})=(x,a_1,\ldots,a_{k-2})$, where $A=\{a_1,\ldots,a_{k-2}\}$ is a subset of $S$ and $x$ is a zero of $f_A(X)$, for some subset $A$ of $S$. By Theorem~\ref{hyp}, the vector $(z_1,\ldots,z_k)$ is a zero of $\phi_S$ and if $q$ is odd, it is a zero of multiplicity two on the line dual to the subspace spanned by $A$. This is precisely the properties that the hypersurface constructed by Blokhuis, Bruen and Thas in \cite{BBT1990} has. Therefore, we have an explicit description of this hypersurface in terms of the $\alpha_C$'s.

\begin{theorem} \label{hyp}
For any subset $A=\{a_1,\ldots,a_{k-2} \}$ of $S$ of size $k-2$
$$
\phi_S(X,a_1,\ldots,a_{k-2})=\alpha_A f_A(X),
$$
if $q$ is even and
$$
\phi_S(X,a_1,\ldots,a_{k-2})=\alpha_A^2 f_A(X)^2,
$$
if $q$ is odd.
\end{theorem}

\begin{proof}
Suppose $q$ is even. If $A$ is a subset of $E$ then
$$
\phi_S(X,a_1,\ldots,a_{k-2})=\sum_{e \in E \setminus A} \alpha_{A \cup \{e\} } \prod_{z \in E \setminus (A \cup \{ e \})} \frac{\det(z,X,A)}{\det(z,e,A)}.
$$
By Lemma~\ref{atoc}, $\alpha_{A \cup \{e\} }=\alpha_A f_A(e)$. Therefore, by Lemma~\ref{interp},
$$
\phi_S(X,a_1,\ldots,a_{k-2})=\alpha_A f_A(X).
$$

Suppose $|A \cap E|=k-2-j$. 

The above proves the theorem for $j=0$ and now we proceed by induction on $j$.

Let $x,y \in E\setminus A$ and $a \in A \setminus E$. Since $q$ is even, $\phi_S(x,A)=\phi_S(a,(A \setminus \{a \}) \cup \{x\})$. 
By induction,
$$
\phi_S(a,(A \setminus \{a \}) \cup \{x\})=\alpha_{(A \setminus \{a \}) \cup \{x\}}f_{(A \setminus \{a \}) \cup \{x\}}(a)=\alpha_{A \cup \{ x\}}=\alpha_A f_A(x),
$$
where the last two equalities use Lemma~\ref{atoc}. Hence,
$$
\phi_S(x,A)=\alpha_A f_A(x),
$$
and
$$
\frac{\phi_S(y,A)}{\phi_S(x,A)}=\frac{f_A(y)}{f_A(x)}.
$$
Thus, the evaluation of 
$$
\frac{f_A(x)}{\phi_S(x,A)} \phi_S(X,A)
$$
is the evaluation of $f_A(X)$ at all points of $E \setminus A$. Now, we argue as in Lemma~\ref{interp}. Both are homogeneous polynomials of degree $t$ which, with respect to a basis containing $A$, are polynomials in two variables. Since they agree at at least $t+1$ linearly independent points they are the same. We have already observed that
$\phi_S(x,A)=\alpha_A f_A(x)$, which completes the proof.

Suppose $q$ is odd. If $A$ is a subset of $E$ then
$$
\phi_S(X,a_1,\ldots,a_{k-2})=\sum_{e \in E \setminus A} \alpha_{A \cup \{e\} }^2 \prod_{z \in E \setminus (A \cup \{ e \})} \frac{\det(z,X,A)}{\det(z,e,A)}.
$$
By Lemma~\ref{atoc}, $\alpha_{A \cup \{e\} }^2=\alpha_A^2 f_A(e)^2$ so, as in Lemma~\ref{interp} but interpolating at $2t+1$ linearly independent points,
$$
\phi_S(X,a_1,\ldots,a_{k-2})=\alpha_A^2 f_A(X)^2.
$$

Suppose $|A \cap E|=k-2-j$. 

The above proves the theorem for $j=0$ and now we proceed by induction on $j$.

Let $x,y \in E\setminus A$ and $a \in A \setminus E$. Arguing as in the $q$ even case,
$$
\phi_S(x,A)=\alpha_A ^2f_A(x)^2,
$$
and so
$$
\frac{\phi_S(y,A)}{\phi_S(x,A)}=\frac{f_A(y)^2}{f_A(x)^2}.
$$
Thus, the evaluation of 
$$
\frac{f_A(x)^2}{\phi_S(x,A)} \phi_S(X,A)
$$
is the evaluation of $f_A(X)^2$ at all points of $E \setminus A$. Now, we argue as in Lemma~\ref{interp}. Both are homogeneous polynomials of degree $2t$ which, with respect to a basis containing $A$, are polynomials in two variables. Since they agree at at least $2t+1$ linearly independent points they are the same. 
\end{proof}

\begin{theorem} \label{forgetqeven}
If $q$ is even then the dimension of the null space of the row space of $\mathrm{M}_{ n}$ is 
$$
{|G|-n-1 \choose k-1}.
$$
Furthermore, the hypotheses in Theorem~\ref{wone}, Theorem~\ref{wtwo} and Theorem~\ref{wtwo2} are never satisfied, apart from in the case $|G|=k+n$ where the hypotheses of Theorem~\ref{wtwo} and Theorem~\ref{wtwo2} are trivially satisfied.
\end{theorem}

\begin{proof}
Let $E$ be a subset of $G$ of size $k+t-1$, where $t=|G|-n-k$. For any choice of $\alpha_C$, where $C$ is a subset of $E$ of size $k-1$, we define
$$
\phi(Y_1,\ldots,Y_{k-1})=\sum_{C \subset E } \alpha_C \prod_{z \in E \setminus C} \frac{\det(z,Y_1,\ldots,Y_{k-1})}{\det(z,C)},
$$
where the sum runs over all subsets of size $k-1$ of $E$.

Define $f_A(X)=\phi(X,A)$ for each subset $A$ of $G$ of size $k-2$, and from this we define $\alpha_C$ for all $C \subset G$, as before. With respect to a basis containing $A$, $f_A(X)$ is a homogeneous polynomial of degree $t$ in two variables, so satisfies Lemma~\ref{interp}. Moreover $f_{D \cup \{ x \}}(y)=\phi(D,x,y)$, so 
$$
\frac{f_{D \cup \{ x \}}(y)}{f_{D \cup \{ y \}}(x)}=1,
$$
so Lemma~\ref{segre1} is also satisfied. Lemma~\ref{theeqn} is derived from these two lemmas, so we conclude that Lemma~\ref{theeqn} holds for these $f_A$. Since Lemma~\ref{theeqn} gives the set of equations defined by $\mathrm{M}_n$ (see Lemma~\ref{upsolution}), the vector $v_G$, whose coordinates are indexed by the subsets $C$ of $G$ of size $k-1$ and whose $C$ coordinate is 
$$
\alpha_C \prod_{z \in G \setminus C} \det(z,C)^{-1},
$$
is in the null space of the row space of $\mathrm{M}_n$.

Thus, the dimension of the null space of the row space of $\mathrm{M}_{ n}$ is at least ${|G|-n-1 \choose k-1}$.

Suppose that $v$ is in the null space of the row space of $\mathrm{M}_{ n}$ and let $\alpha_C$ be defined by
$$
(v)_C=\alpha_C \prod_{z \in G \setminus C} \det(z,C)^{-1}.
$$
Let $E$ be a subset of $G$ of size $k+t-1$ and define
$$
\phi(Y_1,\ldots,Y_{k-1})=\sum_{C } \alpha_C \prod_{z \in E \setminus C} \frac{\det(z,Y_1,\ldots,Y_{k-1})}{\det(z,C)},
$$
where the sum runs over all subsets of size $k-1$ of $E$. For any $C\subset E$, $\phi(C)=\alpha_C$. 

Moreover, since $v$ is in the null space of the row space of $\mathrm{M}_{ n}$ we have that the equation in Lemma~\ref{theeqn} holds. This implies that for any $C$ where $|C \cap E|=k-2$, $\phi(C)=\alpha_C$.
Therefore the $\alpha_C$'s, where $C \subset E$ determine $\alpha_C$ 
where $|C \cap E|=k-2$. Now we can deduce that $\alpha_C$, where $C \subset E'$ is determined by the same $\alpha_C$'s, when $|E' \cap E|=k+t-2$ and extrapolate to deduce that all $\alpha_C$'s are determined by the $\alpha_C$'s, where $C \subset E$.

Thus, the dimension of the null space of the row space of $\mathrm{M}_{ n}$ is at most ${|G|-n-1 \choose k-1}$.

If $\mathrm{M}_{ n}$ has a vector of weight one in its column space then this forces $\alpha_C=0$ for some $C$, which would make the dimension of the row space of $\mathrm{M}_{ n}$ strictly less than ${|G|-n-1 \choose k-1}$. If $\mathrm{M}_{ n}$ has property $W$ then there is a $C$ and a $C'$, subsets of $G$ of size $k-1$, intersecting in $k-2$ vectors, such that $\alpha_C=a \alpha_C'$, for some $a \in {\mathbb F}_q$. This would impose a condition on the null space of the row space and again imply that the dimension of the row space of $\mathrm{M}_{ n}$ is strictly less than ${|G|-n-1 \choose k-1}$.
\end{proof}

\section{Conclusions}

Although the proofs are quite technical, the main results in this article are easily stated and potentially useful. Given an arc $G$ in a space of odd characteristic, one can quickly determine an upper bound on how large an arc one can hope to extend it to. This is done by increasing $n$ one by one to obtain the minimum $n_0$ such that $\mathrm{M}_{ n_0+1}$ has a vector of weight one in its column space. Then Theorem~\ref{wone} implies that $G$ cannot be extended to an arc of size $q+2k+n_0-|G|$. If $(G,n_0)$ satisfies property $W$ and $S$ is an arc of size $q+2k-n_0-1-|G|$ containing $G$ then, by Theorem~\ref{wtwo}, we can determine the co-secants to $S$ containing only points of $G$. Furthermore, if $2n_0 \geqslant |G|-k-1$ then, by Theorem~\ref{wtwo2}, we can determine the algebraic hypersurface $\phi_S$ associated with $S$. 

It should be possible to prove explicit upper bounds if one assumes that $G$ has some structure. For example one might assume that in the planar case when $k$ is three, $G$ is contained in a cubic curve, or all but one or a few of the points of $G$ are contained in a conic. Likewise, for general $k$ one might assume that $G$ is contained in a normal rational curve or all but one or a few of the points of $G$ are contained in a normal rational curve. Using this structure one may then be able to calculate the column space of $\mathrm{M}_{ n}$.  If this is possible, one may also be able to determine precisely the large arcs to which $G$ extends. There should be many results of this type.

\begin{acknowledgements}
I would like to thank Ameera Chowdhury, Jan De Beule and Michel Lavrauw for their comments relating to this manuscript, they were most helpful.
\end{acknowledgements}

\affiliationone{
   Simeon Ball\\
   Departament de Matem\`atiques, \\
Universitat Polit\`ecnica de Catalunya, \\
M\`odul C3, Campus Nord,\\
c/ Jordi Girona 1-3,\\
08034 Barcelona, Spain \\
   \email{simeon@ma4.upc.edu}}

\begin{thebibliography}{99}
%
%
\bibitem{Ball2012}
{\bibname S. Ball}, On sets of vectors of a finite vector space in which every subset of basis size is a basis, {\it J. Eur. Math. Soc.}, {\bf 14} (2012) 733--748. 

\bibitem{Ball2015}
{\bibname S. Ball}, {\em Finite Geometry and Combinatorial Applications}, London Mathematical Society Student Texts {\bf 82}, Cambridge University Press, 2015.

\bibitem{BdB2012}
{\bibname S. Ball \and J. De Beule}, On sets of vectors of a finite vector space in which every subset of basis size is a basis II, {\it Des. Codes Cryptogr.}, {\bf 65} (2012) 5--14.


\bibitem{BBT1990} 
{\bibname A. Blokhuis, A. A. Bruen \and J. A. Thas}, Arcs in $PG(n,q)$, MDS-codes and three fundamental problems of B. Segre - some extensions, {\it Geom. Dedicata}, {\bf 35} (1990) 1--11. 


\bibitem{Chowdhury2015} 
{\bibname A. Chowdhury}, Inclusion Matrices and the MDS Conjecture, {\tt arXiv:1511.03623v2}, 2015. 

\bibitem{Coolsaet2015} 
{\bibname K. Coolsaet}, The complete arcs of $\mathrm{PG}(2,31)$, {\it J. Combin. Designs}, {\bf 23} (2015) 522--533.

\bibitem{CS2011} 
{\bibname K. Coolsaet \and H. Sticker}, The complete $k$-arcs of $\mathrm{PG}(2,27)$ and $\mathrm{PG}(2,29)$, {\it J. Combin. Designs}, {\bf 19} (2011) 111--130.

\bibitem{Frankl1990} 
{\bibname P. Frankl}, Intersection theorems and mod $p$ rank of inclusion matrices, {\it J. Combin. Theory Ser. A}, {\bf 54} (1990) 85--94.


\bibitem{HS2001} 
{\bibname J. W. P. Hirschfeld \and L. Storme}, The packing problem in statistics, coding theory and finite projective spaces: update 2001, in {\it Developments in Mathematics}, {\bf 3}, Kluwer Academic Publishers. {\it Finite Geometries}, Proceedings of the {\it Fourth Isle of Thorns Conference}, pp. 201--246.



\bibitem{Keri2006} 
{\bibname G. K\'eri}, Types of superregular matrices and the number of $n$-arcs and complete $n$-arcs in $\mathrm{PG}(r, q)$, {\it J. Combin. Designs}, {\bf 14} (2006) 363--390.

\bibitem{MS1977} 
{\bibname F. J. MacWilliams \and N. J. A. Sloane}, {\it The Theory of Error-Correcting Codes}, North-Holland, 1977.

\bibitem{Segre1955a} 
{\bibname B. Segre}, Ovals in a finite projective plane, {\it Canad. J. Math.}, {\bf 7} (1955) 414--416.


\bibitem{Segre1967} 
{\bibname B. Segre}, Introduction to Galois geometries, {\it Atti Accad. Naz. Lincei Mem.}, {\bf 8} (1967) 133--236.

\bibitem{Vardy2006} {\bibname A. Vardy}, {\tt http://media.itsoc.org/isit2006/vardy/handout.pdf}, 2006.

\bibitem{Wilson1990} 
{\bibname R. M. Wilson}, A diagonal form for the incidence matrices of $t$-subsets {\it vs.} $k$-subsets, {\it Europ. J. Combinatorics}, {\bf 11} (1990) 609--615.

\end{thebibliography}
\end{document}